\documentclass{article}
\usepackage{amsmath}
\usepackage{amssymb}
\usepackage{amsthm}
\usepackage{bbm,mathtools}

\newcommand{\e}[1]{\prescript{\epsilon\!}{}{#1}}

\newcommand{\mli}{\mathop{\rm \mu\text{-}liminf}}

\def\gph{\mathop{\rm gph}}

\def\reals{\mathbb{R}}

\def\uball{\mathbb{B}}
\def\ereals{\overline{\mathbb{R}}}

\def\interior{\mathop{\rm int}\nolimits}
\def\rinterior{\mathop{\rm rint}}

\def\comp{\raise 1pt \hbox{$\scriptstyle\circ$}}

\def\dom{\mathop{\rm dom}\nolimits}

\def\upto{{\raise 1pt \hbox{$\scriptstyle \,\nearrow\,$}}}
\def\downto{{\raise 1pt \hbox{$\scriptstyle \,\searrow\,$}}}
\def\inte{\mathop{\rm int}}
\def\cl{\mathop{\rm cl}\nolimits}
\def\co{\mathop{\rm co}}

\def\epi{\mathop{\rm epi}}

\def\tos{\rightrightarrows}

\def\B{{\cal B}}

\def\H{{\cal H}}

\def\T{{\cal T}}

\newtheorem{theorem}{Theorem}
\newtheorem{lemma}[theorem]{Lemma}
\newtheorem{corollary}[theorem]{Corollary}

\newtheorem{example}{Example}
\newtheorem{remark}{Remark}
\theoremstyle{definition}
\newtheorem{definition}{Definition}

\theoremstyle{empty}

\begin{document}

\title{Conjugates of integral functionals on continuous functions}
%\title{Convex duality in optimization of BV processes}
\author{Ari-Pekka Perkki\"o\thanks{Department of Mathematics, Ludwig-Maximilians-Universit\"at M\"unchen, Theresienstr. 39, 80333 M\"unchen, Germany. The author is grateful to the V\"ais\"al\"a Foundation for the financial support.}}

\maketitle

\begin{abstract}
This article gives necessary and sufficient conditions for the dual representation of Rockafellar in [Integrals which are convex functionals. II, Pacific J. Math., 39:439-–469, 1971] for integral functionals on the space of continuous functions.
\end{abstract}


\section{Introduction}

Let $ T$ be a $\sigma$-compact locally compact   Hausdorff ($\sigma$-lcH) with the topology $\tau$, the Borel-$\sigma$-algebra $\B( T)$ and a strictly positive regular measure $\mu$. The space of $\reals^d$-valued continuous functions $C:=C( T;\reals^d)$, that vanish at infinity, is a Banach space with respect to the supremum norm. The Banach dual of $C$ can be identified with the space $M$ of $\reals^d$-valued signed regular measures with finite total variation on $ T$ via the bilinear form
\[
\langle y,\theta\rangle := \int yd\theta;
\]
see \cite[Section~7]{fol99}.

This article studies conjugates of convex integral functionals on $C$. Such functionals appear in numerous applications in optimal control, variational problems over functions of bounded variation, optimal transport, financial mathematics, and in extremal value theory. Applications to singular stochastic control and finance are given in \cite{pp15a}.

Our main theorems sharpen those of \cite{roc71} and \cite{per14} by giving necessary and sufficient conditions for the conjugacy of integral functionals on $C$ and functionals on the space of measures with integral representation involving the recession function of the conjugate integrand. For closely related results, we refer to \cite{ab88,bv88} and the references therein. Our characterizations are new in the literature and they complete the idea initiated in \cite{roc71} on expressing the singular parts of the conjugates on the space of measures in terms of recession functions. 

The paper starts by  sharpening \cite[Theorem~6]{roc71} on the integral representation of the support function of continuous selections of a convex-valued mapping. We show that inner semicontinuity of the mapping is not only sufficient but also a necessary condition for the validity of the representation. Section~\ref{sec:main} gives the main results on integral functionals. The conditions appearing in the main results are analyzed further in Section~\ref{sec:suf2}.

\section{Inner semicontinuity and continuous selections}\label{sec:isc}

Let $S$ be a set-valued mapping from $ T$ to $\reals^d$ and denote the set of its continuous selections by
\[
C(S) := \{y\in C\mid y_t\in S_t\ \forall t\in T\}.
\]
This section sharpens \cite[Theorem~6]{roc71} on the integral representation of the {\em support function}
\[
\sigma_{C(S)}(\theta) := \sup_{y\in C(S)}\langle y,\theta\rangle
\]
of $C(S)$. Theorem~6 of \cite{roc71} states that if $S$ is ``inner semicontinuous'' and $ T$ is compact, then
\[
\sigma_{C(S)}(\theta)=\int \sigma_{S}(d\theta/d|\theta|)d|\theta|:=\int \sigma_{S_t}((d\theta/d|\theta|)_t)d|\theta|_t,
\]
where $|\theta|$ denotes the total variation of $\theta$ and $\sigma_{S_t}$ is the normal integrand defined pointwise by
\[
\sigma_{S_t}(x) := \sup_{y\in S_t}x\cdot y.
\]
Theorem~\ref{thm:sf} below extends the result to general $\sigma$-lcH $ T$ and shows that inner semicontinuity of $S$ is a necessary condition for the validity of the integral representation.

The proof of sufficiency is analogous of that of Rockafellar, but we repeat it here for completeness since we do not assume compactness of $ T$. To the best of our knowledge, the necessity of the inner semicontinuity has not been recorded in the literature before.

We denote the neighborhoods of $t\in T$ by $\H_t$ and the open neighborhoods of $x\in\reals^d$ by $\H_x^o$. The mapping $S$ is {\em inner semicontinuous (isc) at $t$} if $S_t \subseteq\liminf S_t$, where
\[
\liminf S_t :=\{x\in\reals^d\mid S^{-1}(A)\in\H_t\ \forall A\in\H^o_x \}.
\]
When $S$ is isc at every $t$, $S$ is {\em inner semicontinuous}. By \cite[Proposition 2.1]{mic56}, $S$ is isc if and only if the {\em preimage} $S^{-1}(O):=\{t\in T \mid S_t\cap O\ne\emptyset\}$ of every open set $O$ is open. Inner semicontinuity is studied in detail in \cite[Section~5]{rw98} in the case where $T$ is Euclidean. We often use Michael's continuous selection theorem \cite[Theorem 3.1''\!']{mic56} that is valid for a Lindel\"of perfectly normal $T_1$-space $T$. We refer to \cite{dug66} for the definitions, and note here that Hausdorff spaces are $T_1$ and $\sigma$-compactness of $ T$ implies that every subset of $ T$ is Lindel\"of, so $ T$ is a Lindel\"of perfectly normal $T_1$-space by \cite[Theorem~XI.6.4 and Problem~VIII.6.7]{dug66}.

Given a set $A$, $N_A(y)$ denotes the set of its {\em normals} at $y$, that is, $\theta\in N_{A}(y)$ if $y\in A$ and 
\[
\langle \theta,y'-y\rangle\le 0 \quad\forall y'\in A,
\]
while $N_A(y):=\emptyset$ for $y\notin A$. For a set-valued mapping $S: T\tos\reals^d$ and a function $y: T\to\reals^d$, we denote the set-valued mapping $t\mapsto N_{S_t}(y_t)$ by $N_S(y)$. The mapping $t\mapsto \cl S_t$ is called the {\em image closure} of $S$.

\begin{theorem}\label{thm:sf}
Assume that $S$ is a convex-valued mapping from $ T$ to $\reals^d$ such that $C(S)\ne\emptyset$. Then 
\[
\sigma_{C(S)}(y)=\int \sigma_{S}(d\theta/d|\theta|)d|\theta|
\]
 if and only if $S$ is inner semicontinuous. In this case $\theta\in N_{C(S)}(y)$ if and only if $d\theta/d|\theta|\in N_S(y)$ $|\theta|$-a.e.
\end{theorem}

\begin{proof}
We first prove sufficiency. The mapping defined by 
\[
\Gamma_t:=\cl\{y_t\mid y\in C(S)\}
\]
is contained in $S$, and, since a mapping is isc if and only if its image closure is isc, $\Gamma$ is isc by part 7 of Theorem~\ref{thm:iop} below. Assuming that $S$ is not inner semicontinuous, there is $t$ such that $\Gamma_t$ is strictly smaller than $S_t$. By a separating hyperplane theorem, there is $x$ with $\sup_{v\in\Gamma_t} x\cdot v < \sup_{v\in S_t} x\cdot v$. Defining $\theta\in M$ so that $|\theta|$ is a Dirac measure at $t$ and $d\theta/d|\theta|=x$, we can write this as
\[
\sup_{y\in C(S)}\langle y,\theta\rangle< \int \sigma_{S}(d\theta/d|\theta|)d|\theta|,
\]
which is a contradiction.

Now we turn to necessity. By the Fenchel inequality, 
\begin{align}\label{eq:fen1}
\langle y,\theta\rangle \le \int \sigma_{S}(d\theta/d|\theta|)d|\theta|
\end{align}
for every $y\in C(S)$, so it suffices to show
\[
\sup_{y\in C(S)} \langle y,\theta\rangle\ge \int \sigma_{S}(d\theta/d|\theta|)d|\theta|.
\]
Denoting $L^\infty(S)=\{w\in L^{\infty}( T,\theta)\mid w_t\in S_t\ \forall t\}$, we have, by \cite[Theorem 14.60]{rw98}, 
\begin{align*}
  \sup_{w\in L^\infty(\theta;S)} \int wd\theta =\int \sigma_{S}(d\theta/d|\theta|)d|\theta|.
\end{align*}
Let $\bar y\in C(D)$,
\[
\alpha< \int \sigma_{S}(d\theta/d|\theta|)d|\theta|
\]
and
$w\in L^\infty(S)$ be such that $\int wd\theta>\alpha$. By regularity of $\theta$ and Lusin's theorem \cite[Theorem 7.10]{fol99}, there is an open $\bar O\subset\T$ such that $\int_{\bar O} (|\bar y|+ |w_t|)d|\theta|<\epsilon/2$, $\bar O^C$ is compact and $w$ is continuous relative to $\bar O^C$. The mapping
\[
\Gamma_t=
\begin{cases}
	w_t\quad&\text{if } t\in \bar O^C\\
	S_t\quad&\text{if }t\in O
\end{cases}
\]
is isc convex closed nonempty-valued so that, by \cite[Theorem 3.1''\!']{mic56}, there is a $y^w\in C$ with $y^w_t=w_t$ on $\bar O^C$ and $y^w_t\in S_t$ for all $t$. Since $|\theta|$ is regular,  there is an open $\hat O\supset \bar O^C$ such that $\int_{\hat O\backslash \bar O^C} |y^w_t|d|\theta|<\epsilon/2$. Since $\bar O^C$ is compact and $T$ is locally compact, we may choose $\hat O$ precompact, by \cite[Theorem XI.6.2]{dug66}. 

Since $\hat O$ and $\bar O$ form an open cover of $ T$ and since $ T$ is normal, there is, by \cite[Theorem 36.1]{mun00}, a continuous partition of unity $(\hat \alpha,\bar\alpha)$ subordinate to $(\hat O,\bar O)$. Defining $y:=\hat \alpha y^w+\bar \alpha \bar y$, we have $y\in C(S)$ and
\begin{align*}
	\int yd\theta &\ge \int_{\bar O^C} wd\theta-\int_{\hat O\backslash \bar O^C} \hat\alpha |y^w|d|\theta|-\int_{\bar O}\bar\alpha |\bar y|d|\theta| \ge\int \alpha-\epsilon,
\end{align*}
which finishes the proof of necessity, since $\alpha<\int \sigma_{S}(d\theta/d|\theta|)d|\theta|$ was arbitrary.

As to the normal vectors, $\theta\in N_{C(D)}(y)$ if and only if  $\sigma_{C(D)}(y)=\langle \theta,y\rangle$ which, by the Fenchel inequality \eqref{eq:fen1}, is equivalent to 
\[
\sigma_{S}(d\theta/d|\theta|)d|\theta| =\int y\cdot (d\theta/d|\theta|) d|\theta|
\]
which is equivalent to condition in the statement, since we also have the pointwise Fenchel inequalities
\[
\sigma_{S_t}(x)\ge y_t\cdot x\quad\forall x\in\reals^d
\]
that are satisfied as equalities if and only if $x\in N_{S_t}(y_t)$.
\end{proof}

The next theorem lists some useful criteria for the preservation of inner semicontinuity in operations. Given a set $A$, $\rinterior A$ denotes its relative interior.

\begin{theorem}\label{thm:iop}
If $S,S^1$ and $S^2$ are isc mappings, then
\begin{enumerate}
\item
$AS$ is isc for any continuous $A:\reals^n\to\reals^m$,
\item
$\co S$ is isc,
\item
$S^1\times S^2$ is isc,
\item
$S^1+S^2$ is isc,
\item
$t\mapsto S^1_t\cap S^2_t$ is isc provided $G:=\{(t,x)\mid x\in S^2_t\}$ is open.
\item
$t\mapsto S^1_t\cap S^2_t$ is isc provided $ T$ is Euclidean and $S^1$ and $S^2$ are convex-valued with $0\in\inte(S^1_t-S^2_t)$ for all $t$,
\item 
Arbitrary pointwise unions of isc mappings are isc.
\end{enumerate}
\end{theorem}

\begin{proof}
1 follows from $(AS)^{-1}(O)=S^{-1}(A^{-1}(O))$, since then preimages of open sets are open. 3 can be derived directly from the definition. Choosing $A(v^1,v^2)=v^1+v^2$ and $S=S^1\times S^2$, 4 follows from 1 and 3. 

To prove 5, let $A\subset\reals^d$ be open, $t\in (S^1\cap S^2)^{-1}(A)$ and $x\in  S^1_t\cap S^2_t\cap A$. The set $G\cap(T\times A)$ is open and contains $(t,x)$, so there is $A'\in\H_x^o$ and $O\in\H_t^o$ such that $A'\subset A$ and $O\times A'\subset G$. Then $O\cap (S^1)^{-1}(A')\subseteq (S^1\cap S^2)^{-1}(A)$, so $(S^1\cap S^2)^{-1}(A)\in \H_t$. Thus $(S^1\cap S^2)^{-1}(A)$ is open. 

To prove 6, let $t^\nu\to t$.  By \cite[Theorem~4.32]{rw98},
\[
\liminf (S^1_{t^\nu}\cap S^2_{t^\nu})\supseteq  \liminf S^1_{t^\nu}\cap \liminf S^2_{t^\nu}.
\]
Taking the intersection over all sequences $t^\nu\to t$ proves the claim, by \cite[Exercise 5.6]{rw98}.

To prove 2, notice that $\co S$ is an union of all the mappings of the form $\sum_{i=1}^m \alpha^i S$ for $\sum_{i=1}^m\alpha^i=1$, $\alpha^i\ge 0$, so the result follows from 1,4 and 7. It remains to prove 7. For any mappings $(S^\alpha)$, $(\bigcup_\alpha S^\alpha)^{-1}(A)= \bigcup ((S^{\alpha})^{-1}(A))$. Thus arbitrary unions of isc mappings is isc, since preimages of open sets are open. 
\end{proof}

\section{Integral functionals of continuous functions}\label{sec:main}

This section studies conjugates of functionals of the form
\[
I_h+\delta_{C(D)}:C\to\ereals,
\]
where $h$ is a convex normal $\B( T)$-integrand on $\reals^d$,
\[
I_h(v) :=\int h_t(v_t)d\mu_t
\]
and $D_t:=\cl\dom h_t$. Here $\dom g :=\{x\mid g(x)<\infty\}$ is the {\em domain of $g$} and $\cl$ denotes the closure operation. Recall that $h$ is a {\em convex normal $\B( T)$-integrand on $\reals^d$} if its {\em epigraphical mapping} $t\mapsto \epi h_t:=\{(v,\alpha)\in\reals^d\times\reals\mid h_t(v)\le \alpha\}$ is closed convex-valued and measurable. A set-valued mapping $S$ is {\em measurable} if preimages of open sets are measurable. We refer the reader to \cite[Chapter~14]{rw98} for a general study of measurable set-valued mappings and normal integrands.

Note that, when $h_t(v)=\delta_{S_t(v)}$, we have $I_h+\delta_{C(D)}=\delta_{C(D)}$ so we are in the setting of Section~\ref{sec:isc}. On the other hand, when $\dom I_h\subseteq C(D)$, we have $I_h+\delta_{C(D)}=I_h$. Sufficient conditions for the above inclusion can be found in \cite{roc71}, \cite{per14} and Section~\ref{sec:suf2} below.
%satisfy this property, so they are covered as a special case of our analysis of $I_h+\delta_{C(D)}$. 

Recall that the {\em conjugate} of an extended real-valued function $F$ on $C$ is the extended real-valued lower semicontinuous (lsc) convex function on $M$ defined by
\[
F^*(\theta):=\sup_{y\in C}\{\langle y,\theta\rangle-F(y)\}.
\]
%When $F$ is proper lsc and convex, the biconjugate theorem \cite[Theorem~5]{roc74} gives $F=(F^*)^*$, a fact which we use throughout the paper.
Rockafellar~\cite{roc71} and more recently Perkki\"o~\cite{per14} gave conditions for the validity of the integral representation
\begin{equation*}
(I_h)^*=J_{h^*},
\end{equation*}
where the functional $J_{h^*}:M\to\ereals$ is defined by
\[
J_{h^*}(\theta)=\int h^*_t((d\theta^a/d\mu)_t)d\mu_t+\int (h^*_t)^\infty((d\theta^s/d|\theta^s|)_t)d|\theta^s|_t,
\]
where $\theta^a$ and $\theta^s$ denote the absolutely continuous and the singular part, respectively, of $\theta$ with respect to $\mu$, and $(h_t^*)^\infty$ denotes the recession function of $h^*_t$. Recall that the {\em recession function} of a proper lsc convex function $g$ is given by
\[
g^\infty(x) = \sup_{\alpha>0}\frac{g(\alpha x+\bar x)-g(\bar x)}{\alpha},
\]
where $\bar x\in\dom g$ is arbitrary; see \cite[Chapter~8]{roc70a}. We often use the identity 
\begin{equation}\label{eq:recsup}
g^\infty=\sigma_{\dom g^*}
\end{equation}
valid for proper lsc convex functions on locally convex topological vector spaces; see \cite[Corollary~3D]{roc66} or \cite[Theorem~13.3]{roc70a} for the finite dimensional case. The identity \eqref{eq:recsup} makes it clear that the properties of $J_{h^*}$ are related to those of $t\mapsto D_t$. 

The following lemma is the first part of the proof of \cite[Theorem~3]{per14}. It shows that the functional $I_h+\delta_{C(D)}$ arises naturally as the conjugate of $J_{h^*}$.

\begin{lemma}\label{lem:J_h}
If $\dom J_{h^*}\ne\emptyset$, then $J_{h^*}^*=I_h+\delta_{C(D)}$.
\end{lemma}

\begin{proof}
We have
 \begin{align*}
	J_{h^*}^*(y)&=\sup_{\theta\in M}\left\{\int yd\theta-J_{h^*}(\theta)\right\}\nonumber \\
	&=\sup_{\theta'\in L^1(\mu;\reals^d)}\left\{\int  y\cdot \theta'd\mu-I_{h^*}(\theta')\right\}\nonumber \\
	&\quad+\sup_{\theta\in M,w\in L^1(|\theta^s|;\reals^d)}\left\{\int y\cdot wd|\theta^s|-\int (h^*)^\infty(w)d|\theta^s|\right\}\nonumber \\
	&=\begin{cases}
 	I_h(y)\quad&\text{if } y_t\in\cl\dom h_t\ \forall t,\\
	+\infty\quad&\text{otherwise}.
	\end{cases}
\end{align*}
Here the second equality follows from the positive homogeneity of  $(h^*_t)^\infty$, and the third follows by first applying \cite[Theorem~14.60]{rw98} on the second and the third line, where one uses \eqref{eq:recsup} and then takes the supremum over all purely atomic finite measures which are singular with respect to $\mu$.
\end{proof}

The next lemma combines \cite[Theorem~4]{roc71} with the condition $C(D)=\cl(\dom I_h\cap C(D))$. This condition is part of the sufficient conditions in the theorem below, which will be reduced to the situation in the lemma. We denote the closed ball centered at $v$ with radius $r$ by $\uball_r(v)$ and the interior of a set $A$ by $\interior A$. When $v=0$, we write simply $\uball_r$.

\begin{lemma}\label{lem:roc}
Assume that $D$ is isc, $C(D)=\cl(\dom I_h\cap C(D))$ and that $y\in \dom I_h$ and $r>0$ are such that $t\mapsto h_t(y_t+v)$ is finite and belongs to $L^1$ whenever $v\in\uball_r$. Then $(I_h+\delta_{C(D)})^*=J_{h^*}$.
\end{lemma}
\begin{proof}
By \cite[Theorem~4]{roc71}, $y\in\interior\dom I_h\cap C(D)$ and
\begin{align*}
I_h^*(\theta) &=\min_{\theta'}\{I_{h^*}(d\theta'/d\mu)+\sigma_{\dom I_h}(\theta-\theta')\mid \theta'\ll \mu\}.
\end{align*}
Since $\interior\dom I_h\cap C(D)\ne\emptyset$, we may apply Fenchel duality \cite[Theorem 20]{roc74} to the sum $I_h+\delta_{C(D)}$. Thus
\begin{align*}
&(I_h+\delta_{C(D)})^*(\theta)\\
&=\min_{\theta''}\{\min_{\theta'}\{I_{h^*}(d\theta'/d\mu)+\sigma_{\dom I_h}(\theta-\theta'-\theta'')\mid \theta'\ll \mu\}+\sigma_{C(D)}(\theta'')\}\\
&=\min_{\theta'}\min_{\theta''}\{I_{h^*}(d\theta'/d\mu)+\sigma_{\dom I_h}(\theta-\theta'-\theta'')+\sigma_{C(D)}(\theta'')\mid \theta'\ll \mu\}\\
&=\min_{\theta'}\{I_{h^*}(d\theta'/d\mu)+\sigma_{\dom I_h\cap C(D)}(\theta-\theta')\mid \theta'\ll \mu\}\\
&=\min_{\theta'}\{I_{h^*}(d\theta'/d\mu)+J_{\sigma_{D}}(\theta-\theta')\mid \theta'\ll \mu\}\\
&=\min_{\theta'}\{\int h^*(d\theta'/d\mu)d\mu+\int (h^*)^\infty(d(\theta-\theta')/d\mu)d\mu\}\\
&\quad+\int (h^*)^\infty(d(\theta^s)/d|\theta^s|)d|\theta^s|.
\end{align*}
Here the third equality follows from another application of Fenchel duality, the fourth from the assumption $C(D)=\cl(\dom I_h\cap C(D))$ and from Theorem~\ref{thm:sf}. By \cite[Corollary 8.5.1]{roc70a}, the last minimum is attained at $d\theta'/d\mu=d\theta/d\mu$, so the last expression equals $J_{h^*}(\theta)$.
\end{proof}

The following theorem and its corollary  are the main results of the paper. The theorem gives necessary and sufficient conditions for the conjugacy of $I_h+\delta_{C(D)}$ and $J_{h^*}$. The corollary sharpens \cite[Theorem~5]{roc71} by giving necessary and sufficient conditions for the conjugacy between $I_h$ and $J_{h^*}$. 

Given a function $g$, $\partial g(y)$ denotes its subgradients at $y$, that is, $\theta\in \partial g(y)$ if
\[
 g(y)+\langle y'-y,\theta\rangle \le g(y')\quad \forall y'.
\]
Given a normal integrand $h$ and $y\in C$, we denote the set-valued mapping $t\mapsto\partial (h_t)(y_t)$ by $\partial h(y)$ and the mapping $t\mapsto N_{D_t}(y_t)$ by $\partial^sh(y)$.
\begin{theorem} \label{thm:main}
Assuming $I_h+\delta_{C(D)}$ and $J_{h^*}$ are proper, they are conjugates of each other if and only if $\dom h$ is isc and $C(D)=\cl(\dom I_h\cap C(D))$, and then $\theta\in\partial (I_h+\delta_{C(D)})(y)$ if and only if
\begin{align}\label{eq:sd}
\begin{split}
d\theta^a/d\mu&\in\partial h(y)\quad\mu\text{-a.e.},\\
d\theta^s/d|\theta^s| &\in\partial^sh(y)\quad|\theta^s|\text{-a.e.}
\end{split}
\end{align}
\end{theorem}

\begin{proof}
To prove the necessity, assume that $I_h+\delta_{C(D)}$ and $J_{h^*}$ are conjugates of each other. By \eqref{eq:recsup}, $\sigma_{\cl\dom (I_h+\delta_{C(D)})}=J_{h^*}^\infty$, where, by the monotone convergence theorem, $J_{h^*}^\infty=J_{(h^*)^\infty}$. Here, by \eqref{eq:recsup} again,  $(h^*)^\infty=\sigma_{D}$. By Lemma~\ref{lem:J_h}, $J_{\sigma_D}^*=\delta_{C(D)}$, so  $\cl(\dom I_h\cap C(D))=C(D)$. In particular, $\sigma_{C(D)}=J_{\sigma_D}$, so $D$ is inner semicontinuous, by Theorem~\ref{thm:sf}, which shows the necessity.

To prove the sufficiency, it suffices, by Lemma~\ref{lem:J_h} and the biconjugate theorem \cite[Theorem~5]{roc74}, to show that $J_{h^*}$ is lsc. For $\epsilon>0$, we define
\[
\e h_t(x) :=\min_{x'\in\uball_\epsilon}h_t(x+x')
\]
and $\e D_t :=\cl\dom \e h_t$. By \cite[Proposition 14.47]{rw98}, $\e h$ is a convex normal integrand, and a direct calculation gives
\[
{\e h}^*_t(v)=h^*_t(v)+\epsilon|v|.
\]
When $\epsilon\searrow 0$,  both $I_{\e h}+\delta_{C(D_\epsilon)}\nearrow I_h+\delta_{C(D)}$ and $J_{\e h^*}\searrow J_{h^*}$ pointwise, so, $J_{h^*}$ is lsc by Lemma~\ref{lem:mon} in the appendix provided that $(I_{\e h}+\delta_{C(D_\epsilon)})^*=J_{(\e h)^*}$. To have this, we will apply Lemma~\ref{lem:roc} to $\e h$. 

We have $\e D=\cl\dom h_t+\uball_\epsilon$, so $\e D$ is isc, by Theorem~\ref{thm:iop}. To show that $C(\e D) =\cl(\dom I_{\e h}\cap C(\e D))$, let  $y\in C(\e D)$. For any $\nu=1,2,\dots$, the mapping $t\mapsto D_t\cap \interior\uball_{\epsilon(1+2^{-\nu-1})}(y_t)$ is convex nonempty-valued and isc, by Theorem~\ref{thm:iop}. By \cite[Theorem 3.1''\!']{mic56} and the assumption $ C(D) =\cl(\dom I_{h}\cap C(D))$, there is $\tilde y^\nu\in \dom I_h\cap C(D)$ such that $\|y-\tilde y^\nu\|\le\epsilon(1+2^{-\nu})$. Defining $y^\nu=(1-\epsilon 2^{-\nu})y+\epsilon 2^{-\nu}y^\nu$, we have $y^\nu \in C(\e D)$ and $y^\nu\rightarrow y$ in $C$. Since $\|y^\nu-\tilde y^\nu\|< \epsilon$, we get $I_{\e h}(y^\nu)\le I_h(\tilde y^\nu)$ and $y^\nu\in\dom I_{\e h}$. Thus $y\in\cl(\dom I_{\e h}\cap C(\e D)$. Finally, since $I_h$ and $J_{h^*}$ are proper, we have, for any $\bar y\in\dom I_h\cap C(D)$, that $t\mapsto \e h_t(\bar y_t+v)$ is finite and integrable whenever $v\in\uball_\epsilon$. Thus $\e h$ satisfies the assumptions of Lemma~\ref{lem:roc} which finishes the proof the sufficiency.

As to the subdifferential formulas, we have, for any $y\in\dom (I_h+\delta_{\cl\dom h})$ and $\theta\in M$, the Fenchel inequalities
\begin{align*}
h(y)+h^*(d\theta^a/d\mu)&\ge y\cdot(d\theta^a/d\mu)\quad\mu\text{-a.e.,}\\
(h^*)^\infty(d\theta^s/d|\theta^s|) &\ge y\cdot (d\theta^s/d|\theta^s|)\quad|\theta^s|\text{-a.e.}
\end{align*}
(the latter holds by \eqref{eq:recsup}), so $\theta\in\partial I_h(v)$ if and only if $I_h(y)+J_{h^*}(\theta)=\langle y,\theta\rangle$ which is equivalent to having the Fenchel inequalities satisfied as equalities which in turn is equivalent to the given pointwise subdifferential conditions.
\end{proof}

\begin{remark}
The properness of $I_h+\delta_{C(D)}$ and $J_{h^*}$ is equivalent to the existence of $\bar x\in L^1(\reals^d)$, $\bar y\in C(D)$ and $\alpha\in L^1$ such that
\begin{align*}
 h_t(v) &\ge v\cdot \bar x_t-\alpha_t,\\
 h^*_t(x) &\ge \bar y_t\cdot x-\alpha_t.
\end{align*}
Indeed, for $y\in\dom I_h\cap C(D)$ and $\theta\in J_{h^*}$, we may choose $\bar y=y$, $\bar x=\theta/d\mu$, and $\alpha:=\max\{h^*(\bar x),h(\bar v)\}$. Choosing $y=\bar y$ and $\theta$ by $d\theta/d\mu=\bar x$ and $\theta^s=0$, gives the other direction. 
\end{remark}

Just like in Theorem~\ref{thm:sf}, the necessity of inner semicontinuity in Theorem~\ref{thm:main} has not been explicitly stated in the literature before. The domain condition will be analyzed more in detail in the next section.

The following corollary extends \cite[Theorem~5]{roc71} and \cite[Theorem~2]{per14} in the case of lcH $\sigma$-compact $ T$. It builds on Theorem~\ref{thm:main} by giving necessary and sufficient conditions for the equality $I_h+\delta_{C(D)}=I_h$.

\begin{corollary}\label{cor:main2}
Assuming $I_h$ and $J_{h^*}$ are proper, they are conjugates of each other if and only if $\dom h$ is isc and $C(D) =\cl\dom I_h$, and then $\theta\in\partial I_h(y)$ if and only if
\begin{align*}%\label{eq:sd}
\begin{split}
d\theta^a/d\mu&\in\partial h(y)\quad\mu\text{-a.e.},\\
d\theta^s/d|\theta^s| &\in\partial^s h(y)\quad|\theta^s|\text{-a.e.}
\end{split}
\end{align*}
\end{corollary}

\begin{proof}
Since $C(D) = \cl\dom I_h$ implies that $I_h=I_h+\delta_{C(D)}$ and that $C(D) = \cl(\dom I_h\cap C(D))$, we have sufficiency by Theorem~\ref{thm:main}. Assuming that $I_h$ and $J_{h^*}$ are conjugates of each other, Lemma~\ref{lem:J_h} implies that $I_h+\delta_{C(D)}=I_h$, so $\dom I_h\cap C(D)=\dom I_h$, and the sufficiency follows from Theorem~\ref{thm:main}.
\end{proof}

The related main theorems by Rockafellar and Perkki\"o in \cite{roc71} and \cite{per14} give sufficient conditions for Corollary~\ref{cor:main2}. Rockafellar's assumption involved the notion of "full lower semicontinuity" of set-valued mappings whereas Perkki\"o formulated the assumption in terms of "outer $\mu$-regularity". We return to these concepts in the next section.

\section{The domain conditions in the main results}\label{sec:suf2}

In this section we will analyze more in detail the conditions 
\begin{align*}
C(D)&=\cl(\dom I_h\cap C(D)),\\
C(D)&=\cl\dom I_h 
\end{align*}
appearing in Theorem~\ref{thm:main} and Corollary~\ref{cor:main2}. The former is characterized in Theorem~\ref{thm:ic} below whereas for the latter we give sufficient conditions in terms of the domain mapping $D$.

The condition in the next result, which extends the  condition in \cite[Theorem~5]{roc71}, can be found in \cite[Proposition~6]{bv88} for metrizable $\sigma$-lcH $T$. Here we generalize to $\sigma$-lcH $T$ and give the converse statement for inner semicontinuous domains.

\begin{theorem}\label{thm:ic}
Let $I_h+\delta_{C(D)}$ and $J_{h^*}$ be proper. If, for every $t\in T$ and every $v\in\rinterior D_t$, there is $O\in\H_t$ and $y\in C(D)$ such that $y_t=v$ and $\int_O h(y)d\mu<\infty$, then 
\[
C(D)=\cl(\dom I_h\cap C(D)).
\]
The converse holds if $D$ is isc.
\end{theorem}
\begin{proof}
Let $y^0\in\dom I_h\cap C(D)$, $\epsilon>0$ and $y\in C(D)$. There is a compact $K\subset T$ such that $\|y^0\|,\|y\|\le \epsilon/2$ outside $K$. By compactness and the assumptions, there exist $t_i\in T$, $i=1,\dots,n$ for which there are $y^i\in C(D)$  and $O_{t_i}\in\H_{t_i}$ such that $|y^i-y|\le \epsilon$ on $O_{t_i}$, $\int_{O_{t_i}} h(y^i)d\mu<\infty$ and $K\subset \bigcup_{i=1}^n O_{t_{i}}$. Let $\alpha^0,\alpha^i$, $i=1,\dots,n$ be continuous partition of unity subordinate to $\{K^C,O_{t_{1}},\dots O_{t_n}\}$. By convexity,  $\sum_{i=0}^n \alpha^i y^{i}\in(\dom I_h\cap C(D))$, and, by construction, $\|\sum_{i=1} \alpha^i y^{t_i}-y\|\le \epsilon$. Since $\epsilon$ was arbitrary, we see that $y\in\cl(\dom I_h\cap C(D))$.

Assume now that $D$ is isc and that $C(D)=\cl(\dom I_h\cap C(D))$. Let $t\in T$ and $v\in\rinterior D_t$. There exists $v^i\in\rinterior D_t$, $i=0,\dots, d$, such that $v\in\co \{v^i\mid i=0,\dots d\}$. Let $\epsilon>0$ be small enough so that $v\in\co \{ \bar v^i\mid i=0,\dots d\}$ whenever, for all $i$, $\bar v^i\in \uball_{v^i,\epsilon}$. By Michael selection theorem \cite[Theorem 3.1''\!']{mic56}, there is, for every $i$, a continuous selection of $D$ taking a value $v^i$ at $t$, so, there exists $y^i\in\dom I_h\cap C(D)$ such that $y^{i}_t\in\uball_{v^i,\epsilon}$. Let $\alpha^i\in(0,1)$ be convex weights such that $\sum_{i=0}^d \alpha^i y^i_t=v$. By convexity, $\sum_{i=0}^d \alpha^i y^i \in C(D)\cap\dom I_h$ which proves the claim.
\end{proof}

We now turn to the condition $C(D)=\cl\dom I_h$ in Corollary~\ref{cor:main2}. We say that a function $y: T\to\reals^d$ is a {\em $\mu$-selection} of $D$ if $y_t\in D_t$ outside a $\mu$-null set. Clearly, every $y\in\dom I_h$ is a continuous $\mu$-selection of $D$. Thus, if every continuous $\mu$-selection of $D$ is a selection of $D$, we have $\dom I_h\subseteq C(D)$ and thus, $I_h+\delta_{C(D)}=I_h$. If, in adddition, $C(D)=\cl(\dom I_h\cap C(D))$, we thus get
\[
C(D)=\cl\dom I_h.
\]

Following \cite{per14}, we say that a set-valued mapping $S$ is {\em outer $\mu$-regular} if
\[
\mli S_t\subseteq \cl S_t\quad\forall t\in T,
\]
where
\[
\mli S_t :=\{x\in\reals^d\mid \exists O\in \H_t:\ \mu(S^{-1}(A)\cap O)=\mu(O)\ \forall A\in\H^o_x \}.
\]
By \cite[Theorem~1]{per14}, every continuous $\mu$-selection of an outer $\mu$-regular mapping is a selection of its image closure. The necessity of outer $\mu$-regularity is given in \cite[Theorem~2]{per14} for solid convex-valued mappings when the underlying topological space is Lindel\"of perfectly normal $T_1$. Here we give sufficient conditions for continuous $\mu$-selections to be selections in terms of familiar continuity notions of set-valued mappings.

Given a topology $\bar\tau$ on $ T$, we denote the $\bar\tau$ neighborhood-system of $t$ by $\bar\H_t$. The mapping $S$ is {\em $\bar \tau$-outer semicontinuous (osc) at $t$} if $\bar \tau\text{-}\limsup S_t\subseteq S_t$, where
\[
\bar \tau\text{-}\limsup S_t:=\{x\in\reals^d \mid S^{-1}(A)\in \bar\H_t^\#\ \forall A\in\H_x^o\},
\]
and $\bar\H_t^\#:=\{B\subset  T\mid B\cap O\ne\emptyset\ \forall O\in \bar\H_t\}$. We denote $\bar\tau\supseteq\tau$ if $\bar\tau$ is finer than $\tau$, and we say that $\bar\tau$ {\em supports $\mu$ locally at $t$} if there is $O\in\bar {\mathcal H}_t$ such that every $O$-relatively $\bar\tau$-open set has a subset with positive measure.

The following definition of $\mu$-fullness is inspired by the notion of full lower semicontinuity introduced in \cite{roc71} for solid convex-valued mappings. Example~\ref{ex:full} below shows that such mappings are $\mu$-full.
\begin{definition}\label{def:full}
A set-valued mapping $S$ is {\em $\bar\tau$-full locally at $t$} if there is $O\in\bar{\mathcal H}_t$ such that
\[
(\bar\tau\text{-}\liminf(\bar\tau\text{-}\limsup S))_{t'}\subseteq\cl S_{t'}\quad\forall\ t'\in O.
\]
The mapping $S$ is {\em $\mu$-full} if, for every $t$, $S$ is $\bar\tau$-full locally at $t$ with respect to some $\bar\tau\supseteq \tau$ that supports $\mu$ locally at $t$.
\end{definition}

Recall that $S$ is {\em $\bar\tau$-continuous at $t$} if it is both isc and osc at $t$, and that $S$ is $\bar\tau$-continuous locally at $t$ it is continuous at every point on some $\bar\tau$-neighborhood of $t$. Example~\ref{ex:rd} below shows that, in the next definition, it is important to allow the topology to depend on $t$. 

\begin{definition}\label{def:3}
The mapping $S$ is $\mu$-continuous if, for every $t$, its image closure is continuous locally at $t$ with respect to some $\bar\tau\supseteq\tau$ that supports $\mu$ locally at $t$.
\end{definition}

\begin{theorem}\label{thm:or}
Continuous $\mu$-selections of $S$ are selections of the image closure of $S$ under either of the following conditions,
\begin{enumerate}
\item $S$ is $\mu$-continuous,
\item $S$ is $\mu$-full.
\end{enumerate}
\end{theorem}

\begin{proof}
1 follows from 2, since when $S$ is $\bar\tau$-continuous locally at $t$, then $\bar\tau\text{-}\limsup S_t=\cl S_t$, $\bar\tau\text{-}\liminf S_t= \cl S_t$, and $\bar\tau\text{-}\liminf(\bar\tau\text{-}\limsup S)_t=\cl S_t$.

To prove 2, let $y$ be a $\tau$-continuous $\mu$-selection of $S$. Fix $t$, and $\bar\tau$ and $O\in\bar{\mathcal H}_t$ in Definition~\ref{def:full}. We may assume that every $O$-relative $\bar\tau$-open set has a set with positive measure, Then every $\bar\tau$-neighborhood of $t'\in O$ contains a subset with positive measure and, thus, a point $t{''}$ with $y_{t'}\in S_{t{''}}$, and hence $y_{t'}\in\bar\tau\text{-}\limsup S_{t'}$. By continuity again, $y_t\in\bar\tau\text{-}\liminf (\bar\tau\text{-}\limsup S)_t$, so $y_t\in \cl S_t$, which proves 2.
\end{proof}

\begin{example}\label{ex:full}
A fully lower semicontinuous mappings $S$ is $\tau$-full. Indeed, \cite[Lemma~2]{roc71} characterizes full lower semicontinuity by the property that 
\[
\interior G=\interior\cl G
\]
for $G=\{(t,x)\mid x\in\interior S\}$. It suffices to show that for a set-valued mapping $\Gamma$, $\interior \gph{\Gamma} \subseteq \gph{((\liminf \Gamma)^o)}$, where $(\liminf \Gamma)^o_t:=\interior\liminf \Gamma_t$. 

Assume that $(t,x)\in\interior\gph{\Gamma}$.  Then there exists $A\in\H_x^o$ and $O\in\H_t$ such that $O\times A\in\gph{\Gamma}$. Thus for every $x'\in A$ and $t'\in O$, $\Gamma^{-1}(O')\in\H_{x'}$ for every $O'\in\H_{t'}$, which means that $x'\in\liminf \Gamma_{t'}$, and hence $(x,t)\in \gph{((\liminf \Gamma)^o)}$. 
\end{example}

\begin{example}\label{ex:rd}
Let $ T=\reals$ be endowed with the Euclidean topology $\tau$. Assume that $S:\reals\tos \reals^d$ has {\em no removable discontinuities} in the sense that, for every $y\in\reals^d$ and $t$, the distance function $t\mapsto d(y,S_t)$ is either left- or right-continuous locally at $t$. Then $S$ is $\mu$-continuous. Indeed, by \cite[Proposition 5.11]{rw98}, the left-continuity  of $d(y,S)$ is equivalent to $S$ being left-continuous with respect to the topology generated by $\{(s,t]\mid s<t\}$, and the right-continuity is handled symmetrically. 
\end{example}

\section{Appendix}

The following lemma gives a general condition for lower semicontinuity of a convex function on the dual of a Frechet space $V$.
\begin{lemma}\label{lem:mon}
A convex function $g$ on the dual of $V$ is lsc if and only if $g+\sigma_{O}$ is lsc for every element $O$ of a local basis of the origin.
\end{lemma}
\begin{proof}
By Krein-Smulian theorem \cite[Theorem 22.6]{kn76}, it suffices to show that $g$ is lsc relative to every weak$^*$-compact set $K$. Assume to the contrary that there is a net $(x^\nu)$ in $K$ converging to $x$, $\alpha\in\reals$, and $\epsilon>0$ such that $g(x^\nu)\le\alpha$ but $g(x)\ge\alpha+\epsilon$. Let $O\subseteq \frac{\epsilon}{2} K^\circ$ be an element of the local basis, where $K^\circ:=\{v\mid \sup_{x\in K}|\langle v,x\rangle|\le 1\}$ is the polar of $K$. Then $\sigma_{O}\le\sigma_{\frac{\epsilon}{2} K ^\circ} \le \frac{\epsilon}{2}$ on $K$, so
\begin{align*}
 \alpha+\epsilon &\le g(x) \le  (g+\sigma_O)(x)\le \liminf (g+\sigma_{O})(x^\nu) < \alpha+\epsilon
\end{align*}
which is a contradiction.
\end{proof}

\bibliographystyle{alpha}
\bibliography{sp}
\end{document}